\documentclass[12pt]{amsart}
\usepackage{amssymb}
\usepackage{amsmath}
\usepackage{amscd}

\theoremstyle{plain}
\newtheorem{theorem}{Theorem}[section]
\newtheorem{corollary}[theorem]{Corollary}

\newtheorem{lemma}[theorem]{Lemma}
\newtheorem{proposition}[theorem]{Proposition}

\theoremstyle{definition}
\newtheorem{definition}[theorem]{Definition}

\newtheorem{example}[theorem]{Example}

\theoremstyle{remark}
\newtheorem{remark}[theorem]{Remark}

\def\N{{\mathbb N}}

\def\C{{\mathbb C}}
\def\P{{\mathbb P}}
\def\Q{{\mathbb Q}}

\def\cA{{\mathcal A}}
\def\cB{{\mathcal B}}
\def\cC{{\mathcal C}}

\def\cP{{\mathcal P}}

\def\c{{\boldsymbol c}}

\def\e{{\boldsymbol e}}

\def\x{{\boldsymbol x}}

\def\0{{\boldsymbol 0}}
\def\1{{\boldsymbol 1}}

\def\Hom{{\rm Hom}}

\def\End{{\rm End}}

\def\Im{{\rm Im}}
\def\Ker{{\rm Ker}}

\def\Dom{{\rm Dom}}
\def\Indef{{\rm Indef}}

\def\rank{{\rm rank}}

\def\graph{{\rm graph}}
\def\Map{{\rm Map}}

\def\Hinge{{\rm Hinge}}

\def\Det{{\rm Det}}

\makeatletter
\newif\if@borderstar
\def\bordermatrix{\@ifnextchar*{%
 \@borderstartrue\@bordermatrix@i}{\@borderstarfalse\@bordermatrix@i*}%
}
\def\@bordermatrix@i*{\@ifnextchar[{\@bordermatrix@ii}{\@bordermatrix@ii[()]}}
\def\@bordermatrix@ii[#1]#2{%
\begingroup
 \m@th\@tempdima8.75\p@\setbox\z@\vbox{%
 \def\cr{\crcr\noalign{\kern 2\p@\global\let\cr\endline }}%
 \ialign {$##$\hfil\kern 2\p@\kern\@tempdima & \thinspace %
  \hfil $##$\hfil && \quad\hfil $##$\hfil\crcr\omit\strut %
  \hfil\crcr\noalign{\kern -\baselineskip}#2\crcr\omit %
  \strut\cr}}%
 \setbox\tw@\vbox{\unvcopy\z@\global\setbox\@ne\lastbox}%
 \setbox\tw@\hbox{\unhbox\@ne\unskip\global\setbox\@ne\lastbox}%
 \setbox\tw@\hbox{%
  $\kern\wd\@ne\kern -\@tempdima\left\@firstoftwo#1%
  \if@borderstar\kern 2pt\else\kern -\wd\@ne\fi%
 \global\setbox\@ne\vbox{\box\@ne\if@borderstar\else\kern 2\p@\fi}%
 \vcenter{\if@borderstar\else\kern -\ht\@ne\fi%
  \unvbox\z@\kern -\if@borderstar2\fi\baselineskip}%
 \if@borderstar\kern-2\@tempdima\kern2\p@\else\,\fi\right\@secondoftwo#1 $%
 }\null \;\vbox{\kern\ht\@ne\box\tw@}%
\endgroup
}
\makeatother

\begin{document}

\title{Projective Linear Monoids and Hinges}
\author{Mutsumi Saito}
\address{Department of Mathematics,
Hokkaido University, Sapporo, 060-0810, Japan}


\maketitle

\begin{abstract}
Let $V$ be a complex vector space.
We propose a compactification $\P M(V)$ of the projective linear group
$PGL(V)$, which can act on the projective space $\P(V)$.
After proving some properties of $\P M(V)$,
we consider its relation to 
Neretin's compactification $\Hinge^*(V)$.

\end{abstract}


\section{Introduction}

From enumerative algebraic geometry and from other motivations,
compactifications of symmetric spaces 
have been studied by many people: Semple \cite{Semple51},
De Concini-Procesi \cite{Wonderful}, etc.
In \cite{Neretin98},
Neretin defined and studied hinges to
describe the boundaries of the compactifications
in an elementary manner.

Let $V$ be a complex vector space.
The projective linear group $PGL(V)$ naturally acts on the projective
space $\P(V)$.
The compactifications of $PGL(V)$ mentioned above, however, 
do not act on $\P(V)$.
Since there are cases in which
it is useful to consider the limit of actions on $\P(V)$ 
(e.g. see \cite{Saito17}, \cite{Saito-Takeda}),
in this paper we propose another compactification $\P M(V)$ of
$PGL(V)$, which can act on $\P(V)$.

Our construction is very similar to Neretin's $\Hinge^*(V)$;
indeed we have a surjective continuous map from an open dense set
of $\P M(V)$ to $\Hinge^*(V)$.
However the monoid structures of $\P M(V)$ and that of a variant of $\Hinge^*(V)$
($\Hinge^*(V)$ is not a monoid) are different.
$\P M(V)$ has the monoid structure compatible with the one
in $\Map(\P(V))$, the set of mappings from $\P(V)$
into itself, while the monoid structure of
a variant of $\Hinge^*(V)$ is compatible
with the one in $\prod_{i=1}^{n-1}\End(\bigwedge^i V)$,
where $n=\dim V$.

This paper is organized as follows.
In Section 2, we define $\P M(V)$ and some other related sets.
In Section 3,
we introduce a topology
in $\P M(V)$ and
prove some expected topological properties.
To prove the compactness of $\P M(V)$ in Section 5,
we recall nets following \cite{Nagata} in Section 4.
We devote Section 5 to the compactness of $\P M(V)$.
By showing that every maximal net of $\P M(V)$ converges,
we prove that
$\P M(V)$ is compact (Theorem \ref{thm:compact}),
which is the main theorem in this paper.
In Section 6, we consider the monoid structure of
$\P M(V)$.
We see that the natural mapping from $\P M(V)$ to
$\Map(\P(V))$ is a monoid-homomorphism 
(Proposition \ref{prop:monoid-hom}).
We thus see that $\P M(V)$ acts on $\P(V)$.
In Sections 7 and 8, we compare $\P M(V)$ with $\Hinge^*(V)$.
We see that we have a surjective continuous map 
from an open dense set
of $\P M(V)$ to $\Hinge^*(V)$
(Theorem \ref{thm:lambdaConti},
Corollary \ref{cor:PM_HtoHingeConti}).

The author would like to thank Professor Hiroshi Yamashita
for his kind suggestions.

\section{Definitions}

Let $V$ be an $n$-dimensional vector space over $\C$.
Set
$$
M:=
\left\{
(A_0,A_1,\ldots, A_m)\,|\,
\begin{array}{l}
m=0,1,2,\ldots\\
0\neq A_i\in \Hom( V_i, V) \quad (0\leq i\leq m), \\
V_0=V,\,   V_{m+1}=0,\\
V_{i+1}=\Ker(A_{i})\quad (0\leq i\leq m) 
\end{array}
\right\}
$$
and
$$
\widetilde{M}
:=
\left\{
(A_0, A_1, \ldots, A_m)\,|\,
\begin{array}{ll}
m=0,1,2,\ldots &\\
A_i\in \End(V) & (\forall i),\\
\cap_{k=0}^{i-1} \Ker(A_k)\not\subseteq \Ker(A_i) & (\forall i),\\
\cap_{k=0}^{m} \Ker(A_k)=0 &
\end{array}
\right\}.
$$
Clearly $m<n=\dim V$.

For $\cA:=(A_0, A_1,\ldots, A_m)\in \widetilde{M}$,
set
$$
K(\cA)_i:=\cap_{k=0}^{i} \Ker(A_k).
$$
Then
$$
\pi(\cA):=(A_0, {A_1}_{|K(\cA)_0},\ldots, {A_m}_{|K(\cA)_{m-1}})\in M.
$$
Clearly $\pi: \widetilde{M}\to M$ is surjective.

Let $\cA:=(A_0, A_1,\ldots, A_m)\in {M}$.
Since $A_i\in \Hom(V_i,V)\setminus\{ 0\}$, 
we may consider the element $\overline{A_i}\in \P \Hom(V_i,V)$
represented by $A_i$, and we can define
$$
\P\cA:=(\overline{A_0}, \overline{A_1},\ldots, \overline{A_m}).
$$
Let $\P M=\P M(V)$ denote the image of $M$ under $\P$.

Let ${\rm Map}(\P(V))$ denote the set of mappings from $\P(V)$
into $\P(V)$.
Then each $\P\cA\in \P M$ defines $\Phi(\P \cA)\in {\rm Map}(\P(V))$
by
$$
\Phi(\P\cA)(\overline{x})=
\overline{A_i x}
\quad
\text{if $x\in V_i\setminus V_{i+1}$},
$$
where
$\overline{x}\in \P(V)$ is the element represented by $x\in V$.

We have the mappings:
\begin{equation}
\widetilde{M}\overset{\pi}{\longrightarrow}
M
\overset{\P}{\longrightarrow}
\P M
\overset{\Phi}{\longrightarrow}
{\rm Map}(\P(V)).
\end{equation}

Clearly $GL(V)\subseteq M, \widetilde{M}$ as a one-term sequence,
 and
$PGL(V)\subseteq \P M, \Phi(\P M)$.

\begin{example}
\label{ex:MneqMtilde}
Let $V=\C^2$, and let
$$
A:=
\begin{bmatrix}
1 & 0\\
0 & 0
\end{bmatrix},
\quad
B:=
\begin{bmatrix}
0 & 1\\
0 & 0
\end{bmatrix}.
$$
Then
$\P\pi((A, B))\neq \P\pi((B,A))$ in $\P{M}$,
but
$\Phi\P\pi((A, B))= \Phi\P\pi((B,A))$ in ${\rm Map}(\P(V))$.
\end{example}

\section{Topology}

We fix a Hermitian inner product on $V$.
Let $W$ be a subspace of $V$.
Via this inner product, 
by considering $V=W\oplus W^\perp$, we regard 
$\Hom(W,V)$ as a subspace of $\End(V)$.

We consider the classical topology 
in $\P\Hom(W,V)$ for any subspace $W$ of $V$.

Let $\cA=(A_0,A_1,\ldots, A_m)\in {M}$.
Then
$A_{i}\in \Hom(V_i, V)$, where $V_i=V(\cA)_i=\Ker(A_{i-1})$.

Let $U_i$ be a neighborhood of $\overline{A_i}$
in $\P\Hom(V(\cA)_i, V)$.
Then set
$$
U_{\P\cA}(U_0,\ldots,U_m)
:=
\left\{ \P\cB\in \P M\,|\,
\begin{array}{l}
\text{for every $i$, there exists $j$ such}\\
\text{that $V(\cB)_j\supseteq V(\cA)_i$ and}\\
\text{$\overline{{B_j}_{|V(\cA)_i}} \in U_i$}
\end{array}
\right\}.
$$

\begin{lemma}
\label{rem:Uniqueness}
Let $\P\cB\in U_{\P\cA}(U_0,\ldots,U_m)$.
Then, for every $i$, there exists {\it unique} $j$ such
that $V(\cB)_j\supseteq V(\cA)_i$ and $\overline{{B_j}_{|V(\cA)_i}} \in U_i$.
Indeed, $j$ must be the maximal $j$ satisfying
$V(\cB)_j\supseteq V(\cA)_i$.
\end{lemma}

\begin{proof}
Suppose that there exist $j_1, j_2$ with $j_1<j_2$ such that 
$V(\cB)_{j_1}, V(\cB)_{j_2}\supseteq V(\cA)_i$ and 
$\overline{{B_{j_1}}_{|V(\cA)_i}} \in U_i$.

Since $j_1<j_2$, we have ${B_{j_1}}_{|V(\cB)_{j_2}}=0$.
Hence ${B_{j_1}}_{|V(\cA)_i}=0$, which contradicts the fact that
$\overline{{B_{j_1}}_{|V(\cA)_i}} \in U_i$.
\end{proof}

\begin{lemma}
\label{lem:TopPM}
The set 
$$
\{ U_{\P\cA}(U_0,\ldots,U_m)\,|\,
\text{$U_i$ is a neighborhood of 
$\overline{A_i}$ $(0\leq i\leq m)$}\}
$$
satisfies the axiom of a base of neighborhoods
of $\P\cA$, and
hence defines a topology in $\P M$.
\end{lemma}

\begin{proof}
First we show
$$U_{\P\cA}(U_0,\ldots,U_m)\cap U_{\P\cA}(U'_0,\ldots,U'_m)
=U_{\P\cA}(U_0\cap U'_0,\ldots,U_m\cap U'_m).
$$
Since `$\supseteq$' is clear, we show `$\subseteq$'.

Let $\P\cB\in U_{\P\cA}(U_0,\ldots,U_m)\cap U_{\P\cA}(U'_0,\ldots,U'_m)$.
Take any $0\leq i\leq m$. Then there exist $j$ and $j'$ such that
$V(\cB)_j, V(\cB)_{j'}\supseteq V(\cA)_i$ and
$\overline{{B_j}_{|V(\cA)_i}}\in U_i$,
$\overline{{B_{j'}}_{|V(\cA)_i}}\in U'_i$.
By Lemma \ref{rem:Uniqueness}, we have $j=j'$,
and thus $\P\cB\in U_{\P\cA}(U_0\cap U'_0,\ldots,U_m\cap U'_m)$.

Let $\P\cB\in U_{\P\cA}(U_0,\ldots,U_m)$.
Put
$$
U'_j:=\left\{
\overline{C}\in \P\Hom(V(\cB)_j, V)\,|\,
\begin{array}{l}
\text{$\overline{C_{|V(\cA)_i}}\in U_i$}\\
\text{for all $i$ with $V(\cB)_j\supseteq V(\cA)_i$}\\
\text{and $\overline{{B_j}_{|V(\cA)_i}}\in U_i$}
\end{array}
\right\}.
$$
For subspaces $W'\subseteq W$ of $V$, set
$$
\P\Hom(W, V)_{W'}:=
\{ \overline{B}\in \P\Hom(W, V)\,|\, B_{|W'}\neq 0\},
$$
and let
$$
\rho^W_{W'}: \P\Hom(W, V)_{W'}\to \P\Hom(W', V)
$$
denote the restriction map.

Note that $\P\Hom(W, V)_{W'}$ is open in
$\P\Hom(W, V)$.
We see that $U'_j (\ni \overline{B_j})$ is open since
$U'_j$ is the intersection of
$(\rho^{V(\cB)_j}_{V(\cA)_i})^{-1}(U_i)$'s.
Then we have
$$
U_{\P\cB}(U'_0,U'_1,\ldots)\subseteq U_{\P\cA}(U_0,\ldots,U_m).
$$
\end{proof}

From now on, we consider $\P M$ as a topological space by Lemma \ref{lem:TopPM}.

\begin{lemma}
\label{PGL:RelativeUsual}
The relative topology of $PGL(V)$ induced from the topology of $\P M$
defined in Lemma \ref{lem:TopPM} coincides with the usual topology of $PGL(V)$
{\rm (}the relative topology induced from the topology of 
$\P\End(V)${\rm )}.
\end{lemma}

\begin{proof}
Let $U_0$ be an open subset of $PGL(V)$ in the usual topology, and
let $B\in U_0$. Then $U_0=U_B(U_0)=U_B(U_0)\cap PGL(V)$.
Hence $U_0$ is open in the relative topology as well.

Next consider $U_{\P\cA}(U_0,\ldots,U_m)\cap PGL(V)$.
By definition
$$
U_{\P\cA}(U_0,\ldots,U_m)\cap PGL(V)
=PGL(V)\cap \bigcap_{i=0}^m (\rho^V_{V(\cA)_i})^{-1}(U_i).
$$
Hence $U_{\P\cA}(U_0,\ldots,U_m)\cap PGL(V)$
is open in the usual topology.
\end{proof}

\begin{proposition}
$\P M$ is Hausdorff.
\end{proposition}

\begin{proof}
Let $\cA, \cB\in M$ with $\P\cA\neq \P\cB$.
Then there exists $i$ such that
$\overline{A_j}=\overline{B_j}$ for all $j<i$,
and that $\overline{A_i}\neq \overline{B_i}$.
We have $V(\cA)_i=V(\cB)_i=:V_i$ and 
$\overline{A_i}, \overline{B_i}\in \P\Hom(V_i, V)$.
Since $\P\Hom(V_i, V)$ is Hausdorff, there exist
neighborhoods $U_A$ and $U_B$ of $\overline{A_i}$ and $\overline{B_i}$,
respectively, with $U_A\cap U_B=\emptyset$.

Let $U_j$ be a neighborhood of $\overline{A_j}=\overline{B_j}$
for $j<i$, and consider
$$
U_{\P\cA}(U_0,U_1,\ldots,U_{i-1}, U_A,\ldots)
\cap
U_{\P\cB}(U_0,U_1,\ldots,U_{i-1}, U_B,\ldots).
$$
Suppose that
$$
\P\cC\in U_{\P\cA}(U_0,U_1,\ldots,U_{i-1}, U_A,\ldots)
\cap
U_{\P\cB}(U_0,U_1,\ldots,U_{i-1}, U_B,\ldots).
$$
Then there exist $j,k$ such that
$V(\cC)_j, V(\cC)_k\supseteq V_i$, 
$\overline{{C_j}_{|V_i}}\in U_A$, and
$\overline{{C_k}_{|V_i}}\in U_B$.
By Lemma \ref{rem:Uniqueness}, we have $j=k$,
which contradicts the fact that
$U_A\cap U_B=\emptyset$.
Hence
$$
U_{\P\cA}(U_0,U_1,\ldots,U_{i-1}, U_A,\ldots)
\cap
U_{\P\cB}(U_0,U_1,\ldots,U_{i-1}, U_B,\ldots)
=\emptyset.
$$
\end{proof}

\begin{proposition}
\label{prop:2ndCountability}
$\P{M}$ satisfies the first axiom of countability.
\end{proposition}

\begin{proof}
Let $\P\cA\in \P M$.
Since each $\P\Hom(V(\cA)_i,V)$ 
satisfies the first axiom of countability,
it is clear that $\P\cA$ has a
neighborhood base consisting of countably-many neighborhoods.
\end{proof}

We introduce the weak topology in 
$\Map(\P(V))$.

\begin{definition}[Weak Topology]
Let $f\in \Map(\P(V))$.
For each $\overline{x_1}, \overline{x_2},
\ldots, \overline{x_m}\in \P(V)$, 
take a neighborhood $U_i$ of $f(\overline{x_i})$.
Put
$$
U_f( \overline{x_1},\ldots, \overline{x_m};
 U_1, U_2,\ldots, U_m):=\{
g\in \Map(\P(V))\,|\,
g(\overline{x_i})\in U_i\quad (\forall i)\}.
$$
We introduce the topology in $\Map(\P(V))$
with
$$
\{ U_f(\overline{x_1},\ldots, \overline{x_m};
 U_1, U_2,\ldots, U_m)\}
$$
as a base of neighborhoods of $f$.
\end{definition}

Clearly $\Map(\P(V))$ is Hausdorff.

\begin{proposition}
\label{prop:PhiPisConti}
$\Phi : \P M\to \Map(\P(V))$ is continuous.
\end{proposition}

\begin{proof}
Let $\P \cA\in \P M$, and
$\overline{x}\in \P V(\cA)_i\setminus \P V(\cA)_{i+1}$.
Then
$$
\Phi(\P\cA)(\overline{x})=\overline{A_i x}.
$$
Let $U$ be a neighborhood of $\overline{A_i x}$ in $\P(V)$.
Put
$$
U_i:=
\{ \overline{C}\in \P\Hom(V(\cA)_i,V)\,|\, 
 \overline{Cx}\in U\}.
$$
Then $U_i$ is a neighborhood of $\overline{A_i}$.
Let $U_k$ $(k\neq i)$ be any neighborhood of $\overline{A_k}
\in \P\Hom(V(\cA)_k,V)$.

Let $\P\cB\in U_{\P\cA}(U_0,\ldots,U_m)$.
Then there exists $j$ such that
$\overline{{B_j}_{|V(\cA)_i}}\in U_i$.
Then $\overline{x}\in V(\cA)_i\subseteq V(\cB)_j$,
and $\overline{x}\notin V(\cB)_{j+1}$ as $\overline{B_j x}\in U$.
Hence
$$
\Phi(\P\cB)(\overline{x})=\overline{B_j x}\in U,
$$
and thus
$$
\Phi(U_{\P\cA}(U_0,\ldots,U_m))
\subseteq
U_{\Phi(\P\cA)}(\overline{x}, U).
$$
\end{proof}

\begin{remark}
By Proposition \ref{prop:PhiPisConti}, 
for any $\overline{x}\in \P(V)$,
the map
$$
\P M\ni \P\cA\mapsto \Phi(\P\cA)(\overline{x})\in \P(V)
$$
is continuous.
\end{remark}

\begin{proposition}
\label{prop:dense}
$PGL(V)$ is dense in $\P{M}$ and $\Phi(\P M)$.
\end{proposition}

\begin{proof}
By Proposition \ref{prop:PhiPisConti},
it is enough to prove that $PGL(V)$ is dense in $\P{M}$.

Let $\cA=(A_0, \ldots, A_m)\in \widetilde{M}$.
If necessary, add $A_{m+1}\epsilon^{m+1}+\cdots+A_l\epsilon^l$
for some $A_{m+1},\ldots, A_l\in \End(V)$
such that
$$
A(\epsilon):=\sum_{i=0}^l A_i \epsilon^i \in GL(V)
$$
for sufficiently small $\epsilon>0$.
We have
\begin{eqnarray*}
\overline{A(\epsilon)_{|K(\cA)_{k-1}}}
&=&
\overline{\epsilon^k({A_k}_{|K(\cA)_{k-1}}+\epsilon 
{A_{k+1}}_{|K(\cA)_{k-1}}+\cdots)}\\
&=&
\overline{{A_k}_{|K(\cA)_{k-1}}+\epsilon 
{A_{k+1}}_{|K(\cA)_{k-1}}+\cdots}.
\end{eqnarray*}
Hence
$\P A(\epsilon)\to \P\pi\cA\quad (\epsilon\to 0)$.
\end{proof}

\begin{proposition}
$PGL(V)$ is open in $\P{M}$ and $\Phi(\P M)$.
\end{proposition}

\begin{proof}
Let $A\in GL(V)$.

Note that
$PGL(V)$
is open in $\P\End(V)$.
Put $U_0:=PGL(V)$.
Let $\P\cB\in U_{\P A}(U_0)$.
Then $\overline{B_0}\in U_0=PGL(V)$, and hence 
$\P\cB=\overline{B_0}\in PGL(V)$. 
Thus $U_{\P A}(U_0)=PGL(V)$ for $A\in GL(V)$.
Hence $PGL(V)$ is open in $\P{M}$.

Next take a basis $\e_1,\ldots,\e_n$ of $V$, and
regard $\End(V)=M_{n\times n}(\C)$.
Put $\e_{n+1}:=\e_1+\e_2+\cdots+\e_n$, and
$$
E_i:=[\e_1,\ldots, \widehat{\e_i}, \ldots, \e_{n+1}]\quad (1\leq i\leq n+1).
$$
Then $E_i\in GL_n(\C)$.
Since $E:=[\e_1,\e_2,\ldots,\e_{n+1}]$ belongs to the open set
$$
\{ D\in M_{n\times (n+1)}(\C)\,|\,
\text{all $n$-minors are nonzero}
\},
$$
we can take neighborhoods $\e_i\in W_i\subset V=\C^n$
$(i=1,2, \ldots, n+1)$
such that for each $i=1,2, \ldots, n+1$
$$
C_i:=[\c_1,\ldots, \widehat{\c_i}, \ldots, \c_{n+1}]\in GL(V)
$$
for all $\c_j\in W_j$ $(j=1,2, \ldots, n+1)$.

Then,
for each $i$,
vectors $\c_1,\ldots, \widehat{\c_i}, \ldots, \c_{n+1}$
are linearly independent, and in turn
 $A\c_1,\ldots, \widehat{A\c_i}, \ldots, A\c_{n+1}$
are linearly independent since $A\in GL(V)$.

Put
$$
U_i:=AW_i \quad (1\leq i\leq n+1).
$$
Then $\overline{U_i}$ is a neighborhood of $\overline{A \e_i}$.
Let 
$$
\Phi(\P\cB)\in U_{\Phi(\P A)}( \overline{\e_1},\ldots, \overline{\e_{n+1}}
; \overline{U_1}, \overline{U_2},\ldots, \overline{U_{n+1}}).
$$
We show that $\P\cB=\overline{B_0}\in PGL(V)$.
Put
$$
I_0:=\{ i\in [1, n+1] \,|\, \e_i\notin \Ker(B_0)\}.
$$
Since $\overline{B_0 \e_i}=\Phi(\P\cB)(\overline{\e_i})\in 
\overline{U_i}$
for $i\in I_0$,
vectors
$\overline{B_0 \e_i}$ $(i\in J\subseteq I_0; |J|\leq n)$ are 
linearly independent.
If $|I_0|\geq n$, then we see that $B_0\in GL_n(\C)$ and thus 
$\P\cB=\overline{B_0}\in PGL(V)$.

Suppose that $|I_0|< n$.
Since $\overline{B_0 \e_i}$ $(i\in I_0)$ are 
linearly independent,
$$
V(\cB)_1=\Ker B_0=\langle \e_i\, |\, i\in [1,n+1]\setminus I_0\rangle.
$$
We have $\dim V(\cB)_1=n-\dim \Im(B_0)=n-|I_0|$.
Put
$$
I_1:=\{ i\in [1, n+1]\setminus I_0 \,|\, \e_i\notin \Ker(B_1)\}.
$$
Then
$\overline{B_1 \e_i}=\Phi(\P\cB)(\overline{\e_i})\in 
\overline{U_i}$ $(i\in I_1)$ 
are linearly independent.
As before, we have
$$
V(\cB)_2=\Ker B_1=\langle \e_i\, |\, i\in [1,n+1]\setminus (I_0\sqcup I_1)\rangle,
$$
and $\dim V(\cB)_2=n-|I_0|-\dim \Im(B_1)=n-|I_0|-|I_1|$.
Repeat this process, and eventually
for some $l$ 
$$
I_l=[1,n+1]\setminus (I_0\sqcup I_1\sqcup\cdots\sqcup I_{l-1}),
$$
and $\dim V(\cB)_l=n-\sum_{j=0}^{l-1}|I_j|(=|I_l|-1)$.
Since $\overline{B_l \e_i}=\Phi(\P\cB)(\overline{\e_i})\in 
\overline{U_i}$ $(i\in I_l)$ 
are linearly independent,
this is a contradiction.
We have thus proved that
$$
U_{\Phi(\P A)}( \overline{\e_1},\ldots, \overline{\e_{n+1}}
; \overline{U_1}, \overline{U_2},\ldots, \overline{U_{n+1}})\cap \Phi(\P M)
\subseteq PGL(V).
$$
\end{proof}

\begin{proposition}
The spaces $\P M$ and $\Phi(\P M)$
 are separable.
\end{proposition}

\begin{proof}
We fix an orthonormal basis of $V$, 
we regard $\End (V)=M_{n\times n}(\C)$.
Clearly $PGL_n(\Q(\sqrt{-1}))$ is dense in $PGL(V)$
in the usual topology.
By Lemma \ref{PGL:RelativeUsual}, it is dense in $PGL(V)$
in the relative topology induced from $\P M$.
Hence by Proposition \ref{prop:dense},
$PGL_n(\Q(\sqrt{-1}))$ is dense in $\P M$,
and hence $\P M$ is separable.

By Proposition \ref{prop:PhiPisConti},
$\Phi(\P M)$ is also separable.
\end{proof}

\section{Nets}

In this section, we recall the concept of nets following
\cite{Nagata}.

A set $\Delta$ with
a relation $\preceq$ satisfying
the following is called a {\it directed set}:
\begin{enumerate}
\item
$a\preceq b$ and $b\preceq c$ implies $a\preceq c$.
\item
$a\preceq a$.
\item
for every two elements $a, b\in \Delta$, there exists an element $c\in
\Delta$ such that $a\preceq c$ and $b\preceq c$.
\end{enumerate}

\begin{definition}
Let $\Delta$ be a non-empty directed set and X a topological space.
A {\it net} on $\Delta$ is a map $\phi: \Delta\to X$, which may be regarded as
a subset $\phi(\Delta):=\{ \phi(\delta)\in X\,|\, \delta\in \Delta\}$
of $X$.

A net $\phi$ of $X$ on $\Delta$ is said to be {\it maximal}
(or {\it universal} or an {\it ultranet}),
if for every subset $S$ of $X$ there exists $\delta_1\in \Delta$
such that $\phi(\delta)\in S$ for all $\delta\succeq \delta_1$
({\it eventually in $S$}),
or there exists $\delta_2\in \Delta$
such that $\phi(\delta)\in X\setminus S$ for all $\delta\succeq \delta_2$.

A net $\phi$ of $X$ on $\Delta$ is said to {\it converge} to a point $x\in X$
(denoted by $\lim_{\delta\in \Delta}\phi(\delta)= x$) 
if for every neighborhood $U$ of $x$
there exists $\delta_U\in \Delta$
such that $\phi(\delta)\in U$ for all $\delta\succeq \delta_U$.
\end{definition}

For the following proposition,
see for example 
\cite[Corollary I$\!$I$\!$I.3.D]{Nagata}.

\begin{proposition}
\label{prop:CompactNet}
A topological space $X$ is compact if and only if
every maximal net of $X$ converges.
\end{proposition}

\section{Compactness}

Let $W$ be a subspace of $V$.
Let $\{ A^{(\delta)}\,|\, \delta\in \Delta\}$ 
be a net of $\Hom(W,V)\setminus\{ 0\}$
such that $\{ \overline{A^{(\delta)}}\}$ is a maximal net of 
$\P\Hom(W, V)$.
Since $\P\Hom(W, V)$ is compact,
$\{ \overline{A^{(\delta)}}\}$ converges
by Proposition \ref{prop:CompactNet}.
Hence there exists
$A_0\in \Hom(W, V)\setminus\{ 0\}$ such that
$$
\overline{A_0}=\lim_{\delta\in \Delta}\overline{A^{(\delta)}}.
$$
Let $V_1:=\Ker (A_0)$.
Since $\{ \overline{A^{(\delta)}}\}$ is maximal,
it is eventually in 
$$
\P\Hom(W, V)_{V_1}=
\{ \overline{B}\in \P\Hom(W, V)\,|\, B_{|V_1}\neq 0\}
$$ or
in its complement.

Suppose that
there exists $\delta_1$ such that
$$
\overline{A^{(\delta)}}\in \P\Hom(W, V)_{V_1}
\quad (\forall \delta\succeq \delta_1).
$$
Set for $\delta\succeq\delta_1$
$$
\overline{A^{(\delta)}_1}:=
\overline{A^{(\delta)}_{|V_1}}\in \P\Hom(V_1, V).
$$
Then we claim that
$\{ \overline{A^{(\delta)}_1}\,|\, \delta\succeq\delta_1\}$
is a maximal net of $\P\Hom(V_1, V)$.

Let
$$
\rho^W_{V_1}: \P\Hom(W, V)_{V_1}
\to \P\Hom(V_1, V)
$$
be the restriction map.
Let $S$ be a subset of $\P\Hom(V_1, V)$.
Then
$\overline{A^{(\delta)}}$ is eventually in $(\rho^W_{V_1})^{-1}(S)$
or in its complement.
Hence
$\overline{A^{(\delta)}_1}$ is eventually in $S$
or in its complement.
We have thus proved that 
$\{ \overline{A^{(\delta)}_1}\,|\, \delta\succeq\delta_1\}$
is maximal.

Since $\P\Hom(V_1, V)$ is compact,
there exists $A_1\in \Hom(V_1, V)$
such that
$\overline{A_1}=\lim_{\delta\succeq\delta_1}\overline{A^{(\delta)}_1}$.
Let $V_2:=\Ker A_1$.
Since $\{ \overline{A^{(\delta)}_1}\,|\, \delta\succeq\delta_1\}$
is maximal,
it is eventually in 
$\P\Hom(V_1, V)_{V_2}$
or in its complement.
Suppose that
there exists $\delta_2$ such that
$$
\overline{A^{(\delta)}_1}\in
\P\Hom(V_1,V)_{V_2}
\quad (\forall \delta\succeq \delta_2).
$$
Set for $\delta\succeq\delta_2$
$$
\overline{A^{(\delta)}_2}:=
\overline{A^{(\delta)}_{|V_2}}\in \P\Hom(V_2, V).
$$
Then as before
$\{ \overline{A^{(\delta)}_2}\,|\, \delta\succeq\delta_2\}$
is a maximal net of $\P\Hom(V_2, V)$.

Repeat this process, and we obtain a sequence
$\{ A_i\in \Hom(V_i, V)\,|\, i=0,1,2,\ldots\}$ 
with $V_i=\Ker (A_{i-1})$ and $V_0=W$
satisfying
\begin{equation}
\label{eqn:Asymptotic}
\overline{A_i}=
\lim_{\delta\succeq\delta_i}\overline{{A^{(\delta)}}_{|V_i}}.
\end{equation}

Since the dimension of $V_i$ is strictly decreasing,
eventually there exists $k$ such that
for some $\delta_{k}\in \Delta$
$$
{A_{k-1}^{(\delta)}}_{|V_{k}}=0 
\quad \text{for $\delta\succeq\delta_{k}$}.
$$

We summarize the above arguments in the following lemma.

\begin{lemma}
\label{thm:expansion}
Let $W$ be a subspace of $V$.
Let $\{ \overline{A^{(\delta)}}\}$ be a maximal net of 
$\P\Hom(W,V)$.
Then there exist $k\in \N$, 
subspaces
$W=V_0\supsetneq V_1\supsetneq V_2\supsetneq 
\cdots \supsetneq V_{k}$,
$\delta_1\preceq \delta_2\preceq \cdots\preceq \delta_k$
in $\Delta$,
and $A_i\in \Hom(V_i,V)$ $(i=0,1,\ldots, k-1)$
satisfying
\begin{enumerate}
\item[\rm (1)]
$\overline{A_i}=
\lim_{\delta\succeq\delta_i}\overline{{A^{(\delta)}}_{|V_i}}$
 for $i\leq k-1$,
\item[\rm (2)]
$V_{k}\subseteq \Ker (A^{(\delta)})$ for $\delta\succeq\delta_k$.
\end{enumerate}
\end{lemma}

\begin{theorem}
\label{thm:Sequence}
$\P {M}$ is compact.
\end{theorem}

\begin{proof}
Let 
$$
\{ \P\cB^{(\delta)}=(\overline{B_0^{(\delta)}}, 
\overline{B_1^{(\delta)}}, \ldots)\}
$$ 
be a maximal net in $\P{M}$.

Then we claim that $\{ \overline{B_0^{(\delta)}}\}$
is a maximal net of $\P\End(V)$.

Let $S$ be any subset of $\P\End(V)$.
Since
$\{ \P\cB^{(\delta)}\}$ is maximal,
it is eventually
in 
$$
\{ \P\cA=(\overline{A_0}, \overline{A_1},\ldots)
\in \P M\,|\, \overline{A_0}\in S\}
$$
or in its complement.
This means that
$\{ \overline{B_0^{(\delta)}}\}$
is eventually in $S$ or in its complement.
We have thus proved that $\{ \overline{B_0^{(\delta)}}\}$
is a maximal net of $\P\End(V)$.

Hence by Lemma \ref{thm:expansion}
there exist $k_0\in \N$, 
subspaces
$V=V_0\supsetneq V_1\supsetneq V_2\supsetneq 
\cdots \supsetneq V_{k_0}$,
$\delta_1\preceq \delta_2\preceq \cdots\preceq \delta_{k_0}$
in $\Delta$,
and $B_i\in \Hom(V_i,V)$ $(i=0,1,\ldots, k_0-1)$
satisfying
\begin{enumerate}
\item[\rm (1)]
$\overline{B_i}=
\lim_{\delta\succeq\delta_i}\overline{{{B_0}^{(\delta)}}_{|V_i}}$
 for $i\leq k_0-1$,
\item[\rm (2)]
$V_{k_0}\subseteq \Ker ({B_0}^{(\delta)})$ for 
$\delta\succeq\delta_{k_0}$.
\end{enumerate}

Suppose that $V_{k_0}\neq 0$.
For $\delta\succeq\delta_{k_0}$, since
$V_{k_0}\subseteq \Ker(B_0^{(\delta)})$,
${B_1^{(\delta)}}_{|V_{k_0}}$ is defined.
We know that
$\{ \P\cB^{(\delta)}\}$ is eventually
in
$$
\{ \P\cA=(\overline{A_0}, \overline{A_1},\ldots)
\in \P M\,|\, V_{k_0}\subseteq \Ker(A_0),\,
{A_1}_{|V_{k_0}}\neq 0\}
$$
or in its complement.
Suppose that $\{ \P\cB^{(\delta)}\}$ is eventually
in the complement.
Then, since $V_{k_0}\subseteq \Ker ({B_0}^{(\delta)})$ for 
$\delta\succeq\delta_{k_0}$,
$\{ \P\cB^{(\delta)}\}$ is eventually
in
$\{ \P\cA=(\overline{A_0}, \overline{A_1},\ldots)
\in \P M\,|\, V_{k_0}\subseteq \Ker(A_1)\subseteq \Ker(A_0)\}$.
Then we know that
$\{ \P\cB^{(\delta)}\}$ is eventually
in
$$
\{ \P\cA=(\overline{A_0}, \overline{A_1},\ldots)
\in \P M\,|\, V_{k_0}\subseteq \Ker(A_1)\subseteq \Ker(A_0),\,
{A_2}_{|V_{k_0}}\neq 0\}
$$
or in its complement.
Repeat this, and we see
that there exists $j_1$ such that
$\{ \P\cB^{(\delta)}\}$ is eventually
in
$$
\{ \P\cA=(\overline{A_0}, \overline{A_1},\ldots)
\in \P M\,|\, V_{k_0}\subseteq \Ker(A_{j})\,(j<j_1),\quad
{A_{j_1}}_{|V_{k_0}}\neq 0\}.
$$
Retake $\delta_{k_0}$ such that,
for all $\delta\succeq\delta_{k_0}$,
 $\P\cB^{(\delta)}$ belong to 
$$
\{ \P\cA=(\overline{A_0}, \overline{A_1},\ldots)
\in \P M\,|\, V_{k_0}\subseteq \Ker(A_{j})\,(j<j_1),\quad
{A_{j_1}}_{|V_{k_0}}\neq 0\}.
$$
Then we claim that 
$\{ \overline{{B_{j_1}}^{(\delta)}_{|V_{k_0}}}\}$
is a maximal net of $\P\Hom(V_{k_0}, V)$.
For every $S\subseteq \P\Hom(V_{k_0}, V)$,
we know that
$\{\P\cB^{(\delta)}\}$ is eventually in
$$
\{ \P\cA\in \P M\,|\,
V_{k_0}\subseteq \Ker(A_{j})\,(j<j_1),\quad
\overline{{A_{j_1}}_{|V_{k_0}}}\in S\}
$$
or in its complement.
This means that
$\{ \overline{{B_{j_1}}^{(\delta)}_{|V_{k_0}}}\}$
is eventually in $S$ or in its complement.
Hence $\{ \overline{{B_{j_1}}^{(\delta)}_{|V_{k_0}}}\}$
is a maximal net of $\P\Hom(V_{k_0}, V)$.
Then apply Lemma \ref{thm:expansion}
to $\{ \overline{{B_{j_1}}^{(\delta)}_{|V_{k_0}}}\}$
to see that
there exist $k_1\in \N$, 
subspaces
$V_{k_0}\supsetneq V_{k_0+1}\supsetneq 
\cdots \supsetneq V_{k_1}$,
$(\delta_{k_0}\preceq) \delta_{k_0+1}\preceq \cdots
\preceq \delta_{k_1}$
in $\Delta$,
and $B_i\in \Hom(V_i,V)$ $(i=k_0,k_0+1,\ldots, k_1-1)$
satisfying
\begin{enumerate}
\item[\rm (1)]
$\overline{B_i}=
\lim_{\delta\succeq\delta_i}\overline{{{B_{j_1}}^{(\delta)}}_{|V_i}}$
 for $k_0\leq i< k_1$,
\item[\rm (2)]
$V_{k_1}\subseteq \Ker ({B_{j_1}}^{(\delta)})$ for 
$\delta\succeq\delta_{k_1}$.
\end{enumerate}

Repeat this process.
We eventually have $V_{k_p}=0$.

Suppose that $k_{s-1}\leq k< k_s$.
Then
$$
\lim_{\delta\succeq \delta_k}
\overline{{B_{j_s}^{(\delta)}}_{|V_k}}=\overline{B_k},
$$
where $j_0=0$.
Hence 
$$
\lim_{\delta\succeq\delta_{k_p}}\P{\cB^{(\delta)}}
=( \overline{B_0}, \overline{B_1},
\ldots, \overline{B_{k_p-1}})\in \P {M}.
$$
Hence by Proposition \ref{prop:CompactNet}
$\P M$ is compact.
\end{proof}

\begin{corollary}
\label{thm:compact}
$\Phi\P M$ is compact.
\end{corollary}

\begin{proof}
This follows from Proposition \ref{prop:PhiPisConti}
and 
Theorem \ref{thm:Sequence}.
\end{proof}

\begin{example}
Let 
$$
A_0=
\begin{bmatrix}
1 & 0 & 0\\
0 & 0 & 0\\
0 & 0 & 0\\
\end{bmatrix},
\quad
A_1=
\begin{bmatrix}
0 & 1 & 1\\
0 & 1 & -1\\
0 & 1 & -1\\
\end{bmatrix},
$$
and let $B$ be any $3\times 3$ matrix.
Put $A_\epsilon:=A_0+\epsilon A_1$.
Then
$(A_\epsilon, B)\in \widetilde{M}$
if $B(-2\epsilon\e_1+\e_2+\e_3)\neq\0$.

By the algorithm before Lemma \ref{thm:expansion}, 
we obtain $(\overline{A_0},\overline{{A_1}_{|V_1}})$ from
$\overline{A_\epsilon}$ $(\epsilon\to 0)$
with $V_1=\Ker(A_0)=\C\e_2+\C\e_3$.
Since $V_2=\Ker({A_1}_{|V_1})=0$,
the algorithm in the proof of Theorem \ref{thm:Sequence} ends here.
Hence
$$
\lim_{\epsilon\to 0}\P\pi(A_\epsilon, B)
=(\overline{A_0},\overline{{A_1}_{|V_1}}).
$$
\end{example}

\section{Products}

In order to define a monoid structure on $M, \P M$
compatible with the one in $\Map(\P(V))$,
we first define a monoid structure on $\widetilde{M}'$,
an extension of $\widetilde{M}$.

Set
$$
\widetilde{M}'
:=
\left\{
(A_0, A_1, \ldots, A_m)\,|\,
\begin{array}{ll}
m=0,1,2,\ldots & \\
A_i\in \End(V) & (0\leq i\leq m)\\
\cap_i \Ker(A_i)=0
\end{array}
\right\}.
$$


We define a product in $\widetilde{M}'$.
For $\cA=(A_0, A_1, A_2,\ldots, A_{l-1})$ and 
$\cB=(B_0, B_1, B_2,\ldots, B_{m-1})\in \widetilde{M}'$,
define $\cA\cB$ by
$$
\cA\cB=(A_0B_0, A_1B_0, \ldots, A_{l-1} B_0, A_0B_1,\ldots, A_{l-1} B_1, A_0B_2,\ldots).
$$
It may be convenient to use the polynomial notation;
for $\cA=(A_0, A_1,$ $\ldots, A_{l-1})$, write
$\cA(t)=\sum_{i=0}^{l-1}A_i t^i$.
Then
$$
(\cA\cB)(t)=\sum_{i=0}^{l-1}\sum_{j=0}^{m-1}A_iB_j t^{lj+i}.
$$

\begin{example}
Let $\cA=(A_0, A_1, A_2), \cB=(B_0,B_1)$.
Then
$$
\cA\cB=(A_0B_0, A_1B_0, A_2B_0, A_0B_1, A_1B_1, A_2B_1).
$$
\end{example}

\begin{lemma}
\label{lem:Product}
For $\cA, \cB\in \widetilde{M}'$,
we have
$\cA\cB\in \widetilde{M}'$.
\end{lemma}

\begin{proof}
Clearly we have $(\cA\cB)_i=0$ for $i\gg 0$.

Note that
\begin{equation}
\label{eqn:IntersectionAB}
\cap_i \Ker(A_i B_j)=\Ker(B_j)
\end{equation}
since $\cap_i \Ker(A_i)=0$.
Hence
$$
\cap_{i,j}\Ker(A_iB_j)=\cap_j (\cap_i  \Ker(A_i B_j))
=\cap_j \Ker(B_j)=0.
$$
Therefore $\cA\cB\in \widetilde{M}'$.
\end{proof}

\begin{proposition}
\label{prop:Associativity}
Let $\cA,\cB,\cC\in \widetilde{M}'$.
Then $(\cA\cB)\cC=\cA(\cB\cC)$, and
$\widetilde{M}'$ is a monoid.
\end{proposition}

\begin{proof}
Let
$\cA(t)=\sum_{i=0}^{k-1}A_it^i, \cB(t)=\sum_{j=0}^{l-1}B_jt^j,
\cC(t)=\sum_{h=0}^{m-1}C_ht^h\in \widetilde{M}'$.
Then
\begin{eqnarray*}
(\cA\cB)\cC
&=&
(\sum_{i,j}A_iB_j t^{kj+i})\sum_h C_ht^h\\
&=&
\sum_{i,j, h}A_iB_jC_j t^{klh+kj+i},\\
\cA(\cB\cC)
&=&
(\sum_{i=0}^{k-1}A_it^i)(\sum_{j,h}B_jC_h t^{lh+j})\\
&=&
\sum_{i,j,h} A_iB_jC_j t^{k(lh+j)+i}.
\end{eqnarray*}
Clearly $\cA I=\cA=I\cA$, where $I\in GL(V)$ is the identity $id_V$.
\end{proof}

For $\cA=(A_0, A_1, \ldots)\in \widetilde{M}'$,
take the subsequence
\begin{equation}
\label{eqn:Psi}
\Psi(\cA)=(A_{i_0}, A_{i_1},\ldots)
\end{equation}
such that
$$
K(\cA)_{i_k}=\cdots=K(\cA)_{i_{k+1}-1}\supsetneq K(\cA)_{i_{k+1}}
$$
for all $k$.
Recall that $K(\cA)_j=\bigcap_{i=0}^j \Ker(A_i)$.
Then clearly $\Psi(\cA)\in \widetilde{M}$.

\begin{lemma}
Let $\cA, \cB\in \widetilde{M}'$. 
Then
$\Psi(\Psi(\cA)\Psi(\cB))=\Psi(\cA\cB)$.
\end{lemma}

\begin{proof}
Let $\cA=(A_0, A_1, A_2,\ldots, A_{l-1}),
\cB=(B_0, B_1, B_2,\ldots, B_{m-1})\in \widetilde{M}'$.
Then by \eqref{eqn:IntersectionAB}
$$
K(\cA\cB)_{lj+i-1}=
(\bigcap_{k=0}^{j-1}\Ker(B_k))\cap
(\bigcap_{k=0}^{i-1}\Ker(A_kB_j)).
$$
Suppose that $K(\cA)_{i}=K(\cA)_{i-1}$, or 
$\Ker(A_i)\supseteq K(\cA)_{i-1}=\cap_{k=0}^{i-1}\Ker(A_k)$.
Then $\Ker(A_iB_j)\supseteq \cap_{k=0}^{i-1}\Ker(A_kB_j)$.
Hence $\Ker(A_iB_j)\supseteq K(\cA\cB)_{lj+i-1}$, or
$K(\cA\cB)_{lj+i}=K(\cA\cB)_{lj+i-1}$.

Next suppose that $K(\cB)_{j}=K(\cB)_{j-1}$ or
$\Ker(B_j)\supseteq K(\cB)_{j-1}=\cap_{k=0}^{j-1}\Ker(B_k)$.
Then
$$
K(\cA\cB)_{lj+i-1}\subseteq
\bigcap_{k=0}^{j-1}\Ker(B_k)
\subseteq
\Ker(B_j)
\subseteq
\Ker(A_iB_j),
$$
or
$K(\cA\cB)_{lj+i}=K(\cA\cB)_{lj+i-1}$.
Hence we conclude
$\Psi(\Psi(\cA)\Psi(\cB))=\Psi(\cA\cB)$.
\end{proof}

\begin{corollary}
$\widetilde{M}$ is a monoid.

For $\cA, \cB\in \widetilde{M}$,
by abuse of notation,
$\Psi(\cA\cB)\in \widetilde{M}$ is also
denoted by $\cA\cB$.
\end{corollary}

We define $\pi(\cA):=\pi(\Psi(\cA))$ for $\cA\in \widetilde{M}'$.

\begin{proposition}
\label{prop:Product}
The product in $\widetilde{M}$
induces the one in ${M}$.
Then it induces the one in $\P M$.
\end{proposition}

\begin{proof}
Let $\cA=(A_0, A_1, \ldots, A_{m-1}), \cB=(B_0, B_1, \ldots)
\in \widetilde{M}'$.

Let $\cA'=(A'_0, A'_1, \ldots, A'_{m-1})$
with
$K(\cA)_{l-1}=K(\cA')_{l-1}$ and
${A'_{l}}_{|K(\cA)_{l-1}}={A_{l}}_{|K(\cA)_{l-1}}$ for all $l$.

Then $A_lB_j=(\cA\cB)_{jm+l+1}$.
We have $K(\cA\cB)_{jm+l}\subseteq \cap_{i<l}\Ker(A_iB_j)$.
Note that
$$
x\in \cap_{i<l}\Ker(A_iB_j) \Leftrightarrow 
B_j x\in \cap_{i<l}\Ker(A_i)=K(\cA)_{l-1}=K(\cA')_{l-1}.
$$
Hence $x\in K(\cA\cB)_{jm+l}$ implies $A_lB_j x=A'_lB_j x$.
Namely $(A_lB_j)_{|K(\cA\cB)_{jm+l}} =(A'_lB_j)_{|K(\cA\cB)_{jm+l}}$.
Hence we have $\pi(\cA\cB)=\pi(\cA'\cB)$.

Next let
$\cB'=(B'_0, B'_1, \ldots )$
with
$K(\cB)_{l-1}=K(\cB')_{l-1}$ and
${B'_{l}}_{|K(\cB)_{l-1}}={B_{l}}_{|K(\cB)_{l-1}}$ for all $l$.

Then $A_iB_l=(\cA\cB)_{lm+i+1}$, and
$$
K(\cA\cB)_{lm+i}\subseteq \cap_{j<l}\cap_k \Ker(A_kB_j).
$$
Since
$\cap_{k}\Ker(A_k B_{j})=\Ker B_{j}$ by
$\cap_k \Ker(A_k)=0$,
we have
\begin{equation}
\label{eqn:KABtoKB}
K(\cA\cB)_{lm+i}\subseteq \cap_{j<l} \Ker(B_j)=
K(\cB)_{l-1}=K(\cB')_{l-1}.
\end{equation}
Hence
$(A_iB_l)_{|K(\cA\cB)_{lm+i}}=(A_iB'_l)_{|K(\cA\cB)_{lm+i}}$.
Hence we have $\pi(\cA\cB)=\pi(\cA\cB')$.

We have thus proved that the product in $\widetilde{M}$
induces the one in ${M}$.
It is obvious that it also induces the one in $\P{M}$.
\end{proof}

\begin{proposition}
\label{prop:monoid-hom}
$\Phi:\P{M}\to {\rm Map}(\P(V))$ is a monoid homomorphism,
where ${\rm Map}(\P(V))$ is a monoid with the composition of
mappings as a product.
Hence $\P M$ acts on $\P(V)$.
\end{proposition}

\begin{proof}
Let $\cA=(A_0,\ldots, A_{l-1}), 
\cB=(B_0,\ldots, B_{m-1})\in {M}$.
Suppose that
$A_iB_j=(\cA\cB)_{k}$ and
$x\in K(\cA\cB)_{k-1}\setminus K(\cA\cB)_{k}$.
Then, as in \eqref{eqn:KABtoKB},
$x\in K(\cB)_{j-1}\setminus K(\cB)_{j}$.
Hence
$\Phi(\P\cB) \overline{x}= \overline{B_j x}$.

Since we have $B_j x\in K(\cA)_{i-1}\setminus K(\cA)_{i}$, 
$$
\Phi(\P\cA)(\Phi(\P\cB) \overline{x})=\Phi(\P\cA)(\overline{B_j x})
=\overline{A_iB_j x}.
$$

By $x\in K(\cA\cB)_{k-1}\setminus K(\cA\cB)_{k}$, we have
$$
\Phi((\P\cA)(\P\cB))\overline{x}=\Phi(\P(\cA\cB))\overline{x}=
\overline{(\cA\cB)_{k}x}
=\overline{A_iB_j x}.
$$
Hence
$$
\Phi(\P\cA)(\Phi(\P\cB) \overline{x})=\Phi((\P\cA)(\P\cB))\overline{x}
\quad (\forall x\in \P(V)).
$$
Namely, $\Phi(\P\cA)\Phi(\P\cB)=\Phi((\P\cA)(\P\cB))$.
\end{proof}

\begin{proposition}
$(\P{M})^\times=PGL(V)=(\Phi\P M)^\times$.
\end{proposition}

\begin{proof}
Let $\P\cA\in (\P M)^\times, (\Phi\P M)^\times$.
Then there exist $\cA=(A_0,\ldots, A_{l-1}),
\cB=(B_0,\ldots, B_{m-1})\in \widetilde{M}$ such that
$\Psi(\cA\cB)=I$.

In $\widetilde{M}'$, 
we have $(\cA\cB)_{lj+i}=A_iB_j$.
Since $\Psi(\cA\cB)=I$, there exist $i,j$ such that
$(\cA\cB)_{lj+i}=A_iB_j=I$ and $(\cA\cB)_k=0$ for all $k<lj+i$.
Then $A_i, B_j\in GL(V)$.
We need to prove $i=j=0$.

Suppose that $i>0$.
Then $A_0B_j=0$, and hence $A_0=0$ since $B_j\in GL(V)$.
This contradicts the fact that $\cA\in \widetilde{M}$.

Suppose that $j>0$.
Then $A_iB_0=0$, and hence $B_0=0$ since $A_i\in GL(V)$.
This contradicts the fact that $\cB\in \widetilde{M}$.
\end{proof}

\begin{example}
Let $n=2$.
Let
\begin{eqnarray*}
A_0&=&
\begin{bmatrix}
1 & 0\\
0 & 0
\end{bmatrix},
\qquad
A_1=
\begin{bmatrix}
0 & 1\\
0 & 2
\end{bmatrix},\\
B_0&=&
\begin{bmatrix}
0 & 0\\
1 & 0
\end{bmatrix},
\qquad
B_1=
\begin{bmatrix}
1 & 2\\
0 & 0
\end{bmatrix}.
\end{eqnarray*}
Then
$$
A_0B_0=A_1B_1=
O,
\qquad
A_1B_0=
\begin{bmatrix}
1 & 0\\
2 & 0
\end{bmatrix}
\qquad
A_0B_1=
\begin{bmatrix}
1 & 2\\
0 & 0
\end{bmatrix}.
$$

Put
$$
A(\epsilon):=A_0+\epsilon^2 A_1,\qquad
B(\epsilon):=B_0+\epsilon B_1.
$$
Then
$$
\lim_{\epsilon\to 0}\overline{A(\epsilon)}=\P\pi(A_0,A_1),\qquad
\lim_{\epsilon\to 0}\overline{B(\epsilon)}=\P\pi(B_0,B_1).
$$
Hence
$$
(\lim_{\epsilon\to 0}\overline{A(\epsilon)})
(\lim_{\epsilon\to 0}\overline{B(\epsilon)})
=\P\pi(A_1B_0,A_0B_1).
$$
Since
$$
A(\epsilon)B(\epsilon)=A_0B_0+\epsilon A_0B_1+\epsilon^2A_1B_0+\epsilon^3A_1B_1
=\epsilon(A_0B_1+\epsilon A_1B_0),
$$
we have
$$
\lim_{\epsilon\to 0}\overline{A(\epsilon)B(\epsilon)}
=\P\pi(A_0B_1, A_1B_0).
$$
Hence
$$
(\lim_{\epsilon\to 0}\overline{A(\epsilon)})
(\lim_{\epsilon\to 0}\overline{B(\epsilon)})
\neq
\lim_{\epsilon\to 0}\overline{A(\epsilon)B(\epsilon)},
$$
and thus the product is not continuous.
\end{example}

\begin{remark}
By definition, it is obvious that the product
$$
PGL(V)\times \P M\ni (g, (\overline{A_0},\overline{A_1}, \ldots, 
\overline{A_m}))
\mapsto (g\overline{A_0},g\overline{A_1}, \ldots, g\overline{A_m})
\in \P M
$$
is continuous.
\end{remark}
\section{Hinges}
\label{sec:Hinges}

Let us recall the notion of hinges defined and studied by
Neretin \cite{Neretin98}, \cite{Neretin00}.

Let $P: V\rightrightarrows W$ be a linear relation,
i.e., a linear subspace of $V\oplus W$.
Let $p_1, p_2$ denote the projections from
$V\oplus W$ onto $V$ and $W$, respectively.
Then $\Dom(P):=p_1(P)$ is called
the domain of $P$, while $\Im(P):=p_2(P)$
is called the image of $P$.
In addition,
$\Ker(P):=p_1(P\cap (V\oplus 0))$ is called 
the kernel of $P$, while
$\Indef(P):=p_2(P\cap(0\oplus W))$ is called 
the indefiniteness of $P$.

A hinge
$$
\cP=(P_0,P_1,\ldots,P_m)
$$
is a sequence of linear relations
$P_j: V\rightrightarrows V$ of dimension $n=\dim V$
such that
\begin{eqnarray*}
&&\Ker(P_j)= \Dom(P_{j+1})\qquad (j=0,1,2,\ldots, m-1)\\
&&\Im(P_j)=\Indef(P_{j+1})\qquad (j=0,1,2,\ldots, m-1)\\
&&\Dom(P_0)=V, \\
&&\Im(P_m)=V,\\
&&P_j\neq \Ker(P_j)\oplus \Indef(P_j)\quad (j=0,1,2,\ldots, m).
\end{eqnarray*}

Let $\Hinge(V)$ denote the set of hinges.
Introduce an equivalence relation $\sim$ in $\Hinge(V)$
by
$$
(P_0,P_1,\ldots,P_m) \sim
(c_0P_0,c_1P_1,\ldots,c_mP_m)
\quad (c_0,c_1,\ldots, c_m\in \C^\times).
$$
Then let $\Hinge^*(V)$ denote the quotient space 
$\Hinge(V)/\!\sim$.
Neretin introduced a topology in $\Hinge^*(V)$,
and proved, among others, that
$\Hinge^*(V)$ is an irreducible projective variety,
and $PGL(V)$ is dense open in $\Hinge^*(V)$.

For $\cA=(A_0,A_1,\ldots, A_m)\in \widetilde{M}$ or $M$,
set
$$
\Im (\cA):=\sum_{i=0}^m \Im ({A_i}_{|K(\cA)_{i-1}}),
$$
and
\begin{eqnarray*}
&&\widetilde{M}_H
:=
\{ {\cA}\in \widetilde{M} \,|\,
\Im(\cA)=\oplus_{i=0}^m \Im ({A_i}_{|K(\cA)_{i-1}})=V
\},\\
&&{M}_H
:=
\{ {\cA}\in {M} \,|\,
\Im(\cA)=\oplus_{i=0}^m \Im ({A_i})=V
\}.
\end{eqnarray*}

We have
\begin{eqnarray*}
\dim V
&=&
\sum_{i\geq 0} (\dim K(\cA)_{i-1}-\dim K(\cA)_i)\\
&=&
\sum_{i\geq 0} \dim \Im {A_i}_{|K(\cA)_{i-1}}.
\end{eqnarray*}
Hence
\begin{eqnarray*}
&&\widetilde{M}_{H}=
\{ {\cA}\in \widetilde{M} \,|\,
\Im(\cA)= V
\}
=
\{ {\cA}\in \widetilde{M} \,|\,
\dim \Im(\cA)=\dim V
\},\\
&&{M}_{H}=
\{ {\cA}\in {M} \,|\,
\Im(\cA)= V
\}
=
\{ {\cA}\in {M} \,|\,
\dim \Im(\cA)=\dim V
\}.\\
\end{eqnarray*}

\begin{example}
\label{ex:M_HNonSemiGroup}
Let
$$
A_0=
\begin{bmatrix}
0 & 1\\
0 & 0
\end{bmatrix},
\quad
A_1=
\begin{bmatrix}
0 & 0\\
1 & 0
\end{bmatrix}
$$
and
$$
B_0=
\begin{bmatrix}
1 & 1\\
0 & 0
\end{bmatrix},
\quad
B_1=
\begin{bmatrix}
0 & 0\\
1 & -1
\end{bmatrix}.
$$
Then $(A_0,A_1), (B_0, B_1)\in \widetilde{M}_{H}$.
In $\widetilde{M}'$,
$$
(B_0, B_1)(A_0,A_1)=
(B_0A_0, B_1A_0, B_0A_1, B_1A_1).
$$
We have
$$
B_0A_0
=
\begin{bmatrix}
0 & 1\\
0 & 0
\end{bmatrix},
\quad
B_1A_0
=
\begin{bmatrix}
0 & 0\\
0 & 1
\end{bmatrix},
$$
$$
B_0A_1
=
\begin{bmatrix}
1 & 0\\
0 & 0
\end{bmatrix},
\quad
B_1A_1
=
\begin{bmatrix}
0 & 0\\
-1 & 0
\end{bmatrix}.
$$
Since $\Ker(B_0A_0)=\Ker(B_1A_0)=\langle \e_1\rangle$
and
$\Ker(B_0A_1)=\Ker(B_1A_1)=\langle \e_2\rangle$,
we have
$$
\Psi((B_0, B_1)(A_0,A_1))=
(B_0A_0, B_0A_1).
$$
Note that
$$
\Im (B_0A_0, B_0A_1)=\langle \e_1\rangle\neq V.
$$
Namely $(B_0, B_1)(A_0,A_1)\notin \widetilde{M}_{H}$.
Hence neither $\widetilde{M}_{H}$ nor $M_H$ is a
submonoid of $\widetilde{M}$
or $M$.
\end{example}

\begin{proposition}
\label{prop:PMvsHinge}
Let $\cA=(A_0,A_1,\ldots, A_m)\in \widetilde{M}_H$ or $M_H$.
Put
$$
P_i:=
\{ (x, A_ix)\in V\oplus V\,|\, x\in K(\cA)_{i-1}\}
+(\0\oplus \sum_{k=0}^{i-1}\Im({A_k}_{|K(\cA)_{k-1}})).
$$
Then
$\varphi(\cA):=(P_0, P_1,\ldots, P_m)\in \Hinge(V)$,
and
$\varphi$ is surjective.

Let $\P{M}_H$ be the image of ${M}_H$
under the map $\P: {M}\to \P M$.
Then clearly $\varphi: {M}_H\to \Hinge(V)$
induces the surjective map
$\P\varphi: \P{M}_H\to \Hinge^*(V)$.
\end{proposition}

\begin{proof}
We prove that
$\dim P_i= \dim V=n$ by induction on $i$.

We have $\dim P_0=n$, since $K(\cA)_{-1}=V$.
Let $i\geq 0$.
We have
\begin{eqnarray*}
\dim P_{i+1}
&=& 
\dim K(\cA)_{i} +\sum_{k=0}^{i}\dim \Im({A_k}_{|K(\cA)_{k-1}})\\
&=&
(\dim K(\cA)_{i-1}-\dim \Im {A_{i}}_{|K(\cA)_{i-1}})\\
&&\qquad
+\sum_{k=0}^{i}\dim \Im({A_k}_{|K(\cA)_{k-1}})\\
&=& 
\dim K(\cA)_{i-1} +\sum_{k=0}^{i-1}\dim \Im({A_k}_{|K(\cA)_{k-1}})\\
&=& \dim P_{i}.
\end{eqnarray*}
The other assertions for $\varphi(\cA)$ to be a hinge are obvious.

Next we prove that $\varphi$ is surjective.
Let $(P_0,P_1,\ldots, P_m)\in \Hinge(V)$.
Then
there exist two partial flags:
\begin{eqnarray*}
&&V=\Dom(P_0)\supseteq \Dom(P_1)\supseteq
\cdots \Dom(P_m)\supseteq 0\\
&&0\subseteq \Im(P_0)\subseteq \Im(P_1)\subseteq\cdots
\subseteq \Im(P_m)=V.
\end{eqnarray*}
Put $V_i:=\Dom(P_i)=\Ker(P_{i-1})$.
Fix a subspace $R_i$ of $\Im(P_i)$ such that
$$
\Im(P_i)=R_i\oplus \Im(P_{i-1}).
$$
Then for any $x\in V_i=\Dom(P_i)$
there exists a unique element $A_i(x)\in R_i$
such that $(x, A_i(x))\in P_i$,
since $\Indef(P_i)=\Im(P_{i-1})$.
Then $A_i:V_i\to R_i\subseteq V$ is a linear operator.
It is easy to see that 
$\cA:=(A_0,A_1,\ldots,A_m)\in M_H$ and
$\varphi(\cA)=(P_0,P_1,\ldots, P_m)$.
\end{proof}

\begin{example}
Clearly $GL(V)\subseteq {M}_H$, 
$PGL(V)\subseteq \P{M}_H$, and
$\P\varphi_{|PGL(V)}=id_{PGL(V)}$.
Here note that in $\Hinge(V)$ each $g\in GL(V)$ is identified
with its graph $\{ (\x, g\x)\,|\, \x\in V\}$.
\end{example}

\begin{example}
\label{ex:HingeVSMine}
Let $V=\C^2=\C\e_1\oplus \C\e_2$.
Let
$$
A_0:=
\begin{bmatrix}
1 & 0\\
0 & 0
\end{bmatrix},
\quad
A_1:=
\begin{bmatrix}
0 & 1\\
0 & 0
\end{bmatrix},
\quad
A_2:=
\begin{bmatrix}
0 & 0\\
0 & 1
\end{bmatrix}.
$$
Then $(A_0,A_1)\in \widetilde{M}\setminus \widetilde{M}_H$,
and $\cA:=(A_0, A_2)\in \widetilde{M}_H$.

We have $\varphi(\cA)=(P_0,P_2)$ with
\begin{eqnarray*}
&&P_0=\graph(A_0)=\{ (\x, A_0\x)\}
=\{ (a\e_1+b\e_2, a\e_1)\,|\, a\in \C \}
,\\
&&P_2=\{ c(\e_2, \e_2)\,|\, c\in \C\}
+\{ (0, a\e_1)\,|\, a\in \C\}.
\end{eqnarray*}

Let 
$$
A(\epsilon):=A_0+A_1\epsilon+A_2\epsilon^2
=
\begin{bmatrix}
1 & \epsilon\\
0 & \epsilon^2
\end{bmatrix}.
$$
Then $A(\epsilon)\in GL(V)$ for $\epsilon\neq 0$.
We have
$$
\lim_{\epsilon\to 0}\P\pi(A(\epsilon))=
\P\pi(A_0,{A_1})\quad \text{in $\P M$}
$$
as in the proof of Proposition \ref{prop:dense},
while
\begin{eqnarray*}
\lim_{\epsilon\to 0}\P\varphi\pi(A(\epsilon))
&=&
\lim_{\epsilon\to 0}\P
\{ (a\e_1+b\e_2, (a+b\epsilon)\e_1+b\epsilon^2\e_2)\,|\, a,b\in \C\}\\
&=&
\P(P_0,P_2)
\end{eqnarray*}
in $\Hinge^*(V)$.
\end{example}

\begin{lemma}
\label{lem:PA=PB}
Let $A,B\in \Hom(W,V)\setminus\{ 0\}$, and
$W_1:=\Ker(A), W_2:=\Ker(B)$.
Suppose that $\overline{A}=\overline{B}$ on $\P(W)\setminus \P(W_1)\cup\P(W_2)$.
\begin{enumerate}
\item[\rm (1)]
If $W_1=W_2$, then
there exists $c\neq 0$ such that $A=cB$.
\item[\rm (2)]
If $W_1\neq W_2$, then
$W=W_1+W_2$, $\Im(A)=\Im(B)$, and $\rank(A)=\rank(B)=1$.
\end{enumerate}
\end{lemma}

\begin{proof}
(1)\quad
Let $W=W_1\oplus U$.
Let $\{ y_i\}$ be a basis of $U$.
Note that $A_{|U}$ and $B_{|U}$ are injective.
There exist $c_i\neq 0$ such that $Ay_i=c_iBy_i$.
Let $i\neq j$.
Then $A(y_i+y_j)=c_iBy_i+c_jBy_j$ must be a nonzero multiple
of $B(y_i+y_j)$. Say
$A(y_i+y_j)=c_iBy_i+c_jBy_j=cB(y_i+y_j)$.
Hence $B((c-c_i)y_i+(c-c_j)y_j)=0$.
Since $B_{|U}$ is injective,
we have $c_i=c=c_j$.
Hence $A=cB$.

(2)\quad
Let $W=(W_1+W_2)\oplus U$.
As in (1), we have $A_{|U}=cB_{|U}$ for some $c\neq 0$.
Let $W_1=W_1\cap W_2\oplus W'_1$.
Suppose that $W'_1\neq 0$.
For $0\neq w_1\in W'_1$ and $0\neq u\in U$,
we have
$A(w_1+u)=Au=cBu$.
Since $u\neq 0$, $A(w_1+u)$ must be a nonzero multiple
of $B(w_1+u)$. Say $A(w_1+u)=dB(w_1+u)$.
Then
$B(dw_1+(d-c)u)=0$.
Since $B_{|W'_1+U}$ is injective,
$c=d=0$.
Hence $U=0$, or $W=W_1+W_2$.

Let $W_2=W_1\cap W_2\oplus W'_2$.
Similarly, if $W'_2\neq 0$, then $W=W_1+W_2$.
Since $W_1\neq W_2$, we have $W'_1\neq 0$ or $W'_2\neq 0$, and hence
$W=W_1+W_2$.

For $0\neq w_1\in W'_1$ and $0\neq w_2\in W'_2$,
$A(w_1+w_2)=A(w_2)$ is a nonzero multiple of
$B(w_1+w_2)=B(w_1)$.
Hence $\Im(A)=\Im(B)$, and $\rank(A)=\rank(B)=1$.
\end{proof}

\begin{proposition}
Let $\P\cA\in \P{M}_H$ and
$\P\cB\in \P{M}$ satisfy
$\Phi(\P\cA)=\Phi(\P\cB)$.
Then $\P\cA=\P\cB$.
\end{proposition}

\begin{proof}
Let $\P\cA=(\overline{A_0},\overline{A_1},\ldots)$
and
$\P\cB=(\overline{B_0},\overline{B_1},\ldots)$.
First we prove $\overline{A_0}=\overline{B_0}$.
Suppose the contrary.
Then, by Lemma \ref{lem:PA=PB},
$\Ker(A_0)\neq \Ker(B_0)$,
$\Im(A_0)=\Im(B_0)$, and $\rank(A_0)=\rank(B_0)=1$.
In particular, $\Ker(A_0)\not\subseteq \Ker(B_0)$,
since $\dim\Ker(A_0)=\dim\Ker(B_0)$.

By $\Phi(\P\cA)=\Phi(\P\cB)$ we have
$$
\overline{{B_0}}_{|\Ker(A_0)\setminus \Ker(B_0)\cup \Ker(A_1)}
=\overline{{A_1}}_{|\Ker(A_0)\setminus\Ker(B_0)\cup \Ker(A_1)}.
$$
Hence we have
$\Im(A_0)=\Im(B_0)=\Im(A_1)$,
which contradicts the assumption that $\P\cA\in \P{M}_H$.
We have thus proved $\overline{A_0}=\overline{B_0}$.

Do the same arguments to have $\overline{A_1}=\overline{B_1}$,
and so on.
\end{proof}

\section{the mapping $\lambda$}

Let $\cA=(A_0, A_1, \ldots, A_m)\in M$,
where $A_i\in \Hom(V(\cA)_{i}, V)$.
Let $V(\cA)'_{i}$ be the orthogonal complement of 
$V(\cA)_{i+1}=K(\cA)_{i}$
in $V(\cA)_i=K(\cA)_{i-1}$;
$V(\cA)_{i}=V(\cA)'_{i}\oplus V(\cA)_{i+1}$.

Let $r_i:=\dim(V(\cA)'_{i})=\rank(A_i)$, and $r_{-1}=0$.
For $r_0+r_1+\cdots+r_{j-1}<k\leq r_0+r_1+\cdots+r_j$,
define $\bigwedge^k\cA\in \End(\bigwedge^k V)$ as the
restriction on
$\bigwedge^{r_0} V(\cA)'_0\otimes \bigwedge^{r_1} V(\cA)'_1\otimes
\cdots\otimes \bigwedge^{k-(r_0+\cdots+r_{j-1})} V(\cA)'_j$
of
$$
\bigwedge^{k} ({A_0}_{| V(\cA)'_0}\oplus {A_1}_{| V(\cA)'_1}\oplus
\cdots \oplus {A_m}_{| V(\cA)'_m}),
$$
where we regard
$$
\Hom(\bigwedge^{r_0} V(\cA)'_0\otimes \bigwedge^{r_1} V(\cA)'_1\otimes
\cdots\otimes \bigwedge^{k-(r_0+\cdots+r_{j-1})} V(\cA)'_j, 
\bigwedge^k V)
\subseteq \End(\bigwedge^k V),
$$
since 
$\bigwedge^{r_0} V(\cA)'_0\otimes \bigwedge^{r_1} V(\cA)'_1\otimes
\cdots\otimes \bigwedge^{k-(r_0+\cdots+r_{j-1})} V(\cA)'_j$
is a direct summand of $\bigwedge^k V$
according to the decomposition
$$
V=\bigoplus_{i=0}^m V(\cA)'_i.
$$

In particular, $\bigwedge^n\cA\in \End(\bigwedge^n V)=
\C id_{\bigwedge^n V}$.
We define $\Det(\cA)\in \C$ by
$\bigwedge^n\cA=\Det(\cA)id_{\bigwedge^n V}$.
Clearly $\cA\in M_H$ if and only if $\Det(\cA)\neq 0$.

Define a map $\lambda: M\to
\prod_{k=1}^{n}\End(\bigwedge^k V)$ by
$$
\lambda:M \ni \cA
\mapsto
(\bigwedge^1\cA, \bigwedge^2\cA,\ldots,\bigwedge^{n}\cA)
\in \prod_{k=1}^{n}\End(\bigwedge^k V).
$$

\begin{remark}
Example \ref{ex:M_HNonSemiGroup} shows that
$\lambda$ is not a monoid homomorphism.
Clearly we have
$$
\lambda(g \cA h)=\lambda(g)\lambda(\cA)\lambda(h)
\quad \text{for all $g,h\in GL(V),\, \cA\in M$.}
$$
\end{remark}

Note that $\bigwedge^k\cA\neq 0$ for any $k$ if $\cA\in {M}_H$.
Hence we can define
$$
\overline{\lambda}:\P{M}_H\ni \cA
\mapsto
(\overline{\bigwedge^1\cA}, \overline{\bigwedge^2\cA},\ldots,
\overline{\bigwedge^{n-1}\cA})
\in \prod_{k=1}^{n-1}\P\End(\bigwedge^k V).
$$

\begin{theorem}
\label{thm:lambdaConti}
$\P{M}_H$ is open in $\P{M}$, and
$\overline{\lambda}$ is continuous.
\end{theorem}

\begin{proof}
Let $\cA=(A_0,A_1,\ldots, A_m)\in M_H$.
Let $\cB^{(l)}$ be a sequence in $M$
such that
$\P\cB^{(l)}$ converges to $\P\cA$.

As in the proof of Theorem \ref{thm:Sequence},
we can take a subsequence 
$\P\cB^{(l')}$ converging to
some $\P\cA'=\P(A'_0,A'_1,\ldots, A'_{m'})\in \P M$,
$0=j_0<j_1<\cdots<j_p$,
and $0=k_{-1}< k_0< k_1< \cdots < k_p$
such that for each $s=0,1,\ldots, p$
\begin{enumerate}
\item[(a)]
$
\lim_{l'\to \infty}
\overline{{B_{j_s}}^{(l')}_{|V(\cA')_k}}=
\overline{A'_k}
\qquad (k_{s-1}\leq k< k_s),
$
\item[(b)]
$V(\cA')_{k_s}\subseteq \Ker(B_{j_s}^{(l')})$.
\end{enumerate}

Since $\P\cB^{(l)}$ converges to $\P\cA$,
we have $\P\cA'=\P\cA$ and $k_p=m+1$.
To simplify the notation,
we put
$$
V(i):=V(\cA)_i,\quad
V(i)':=V(\cA)'_i
$$
in this proof.

We thus have
\begin{eqnarray}
&&
\lim_{l\to \infty}
\overline{{B_{j_s}}^{(l)}_{|V(k)}}=
\overline{A_k}
\qquad (k_{s-1}\leq k< k_s),\\
&&
\label{eq:ks-js}
V(k_s)\subseteq \Ker(B_{j_s}^{(l)}).
\end{eqnarray}

Then
the
restrictions on
$$
\bigwedge^{r_0} V(0)'\otimes \bigwedge^{r_1} V(1)'\otimes
\cdots\otimes \bigwedge^{k-(r_0+\cdots+r_{j-1})} V(j)'
$$
of some nonzero scalar multiples of
$$
\bigwedge^{k} ({B^{(l)}_{j_0}}_{| V(0)'\oplus\cdots\oplus 
V(k_0-1)'}\oplus
\cdots \oplus {B^{(l)}_{j_p}}_{| V(k_{p-1})'\oplus\cdots\oplus 
V(k_p-1)'})
$$
converge to $\bigwedge^k\cA$,
the
restriction on
$\bigwedge^{r_0} V(0)'\otimes \bigwedge^{r_1} V(1)'\otimes
\cdots\otimes \bigwedge^{k-(r_0+\cdots+r_{j-1})} V(j)'$
of
$$
\bigwedge^{k} ({A_0}_{| V(0)'}\oplus {A_1}_{| V(1)'}\oplus
\cdots \oplus {A_m}_{| V(m)'}).
$$
Since $\cA\in M_H$, we have $\bigwedge^n\cA\neq 0$.
Hence
$$
\bigwedge^{n} ({B^{(l)}_{j_0}}_{| V(0)'\oplus\cdots
\oplus V(k_0-1)'}\oplus
\cdots \oplus {B^{(l)}_{j_p}}_{| V(k_{p-1})'\oplus\cdots
\oplus V(k_p-1)'})
\neq 0 \quad (l\gg 0).
$$
This means for $l\gg 0$
\begin{equation}
\label{eqn:FullImageBl}
\Im({B^{(l)}_{j_0}}_{| V(0)'\oplus\cdots\oplus V(k_0-1)'}\oplus
\cdots \oplus 
{B^{(l)}_{j_p}}_{| V(k_{p-1})'\oplus\cdots\oplus V(k_p-1)'})=V
\end{equation}
and  
\begin{equation*}
\text{
$B^{(l)}_{j_s}$ are injective on 
$V(k_{s-1})'\oplus\cdots\oplus V(k_s-1)'$}
\end{equation*}
for all $s$.
Since 
\begin{equation}
\label{eqn:VandV'}
V(k_{s-1})=V(k_{s-1})'\oplus
V(k_{s-1}+1)'\oplus\cdots\oplus
V(k_{s}-1)'\oplus V(k_{s}),
\end{equation}
we have by \eqref{eq:ks-js}
\begin{equation}
\label{eq:PreVA=VB}
V(k_s)= \Ker(B^{(l)}_{j_s})\cap V(k_{s-1}) \quad (l\gg 0).
\end{equation}

We claim that
\begin{equation}
\label{eq:VA=VB}
V(k_s)= \Ker(B^{(l)}_{j_s}) \quad (l\gg 0).
\end{equation}
We prove \eqref{eq:VA=VB} by induction on $s$.
Letting $s=0$ in \eqref{eq:PreVA=VB}, we have
$$
V(k_0)= \Ker(B^{(l)}_{j_0})\cap V(k_{-1})
=\Ker(B^{(l)}_{j_0})\cap V=\Ker(B^{(l)}_{j_0}).
$$
For $s>0$, since $\Ker(B^{(l)}_{j_s})\subseteq \Ker(B^{(l)}_{j_{s-1}})$,
we have
$$
V(k_s)= \Ker(B^{(l)}_{j_s})\cap V(k_{s-1})
=\Ker(B^{(l)}_{j_s})\cap \Ker(B^{(l)}_{j_{s-1}})
=\Ker(B^{(l)}_{j_s})
$$
by the induction hypothesis.
We have thus proved the claim \eqref{eq:VA=VB}.

Since $\Ker(B^{(l)}_{j_s})= V(\cB^{(l)})_{j_s+1}$, we have
$$
{B^{(l)}_{j_s+1}}_{|V(k_s)}\neq 0
$$
for $l\gg 0$.
Recalling the choice of $j_{s+1}$ in the proof of Theorem \ref{thm:Sequence},
we see that $j_{s+1}=j_s+1$.
Hence we have $j_s=s$ for $s=0,1,\ldots, p$.

Thus for $l\gg 0$ we have 
\begin{equation}
\label{eqn:VKerBVB}
V(k_{s-1})= \Ker(B^{(l)}_{s-1})=V(\cB^{(l)})_s.
\end{equation}
Hence by \eqref{eqn:VandV'}
\begin{equation}
\Im(B^{(l)}_s)=
\Im({B^{(l)}_s}_{|V(\cA)'_{k_{s-1}}\oplus\cdots\oplus 
V(\cA)'_{k_s-1}}).
\end{equation}
Then from \eqref{eqn:FullImageBl} we obtain
\begin{equation}
\text{$\cB^{(l)}\in M_H$ for $l\gg 0$.}
\end{equation}
Hence, thanks to Proposition \ref{prop:2ndCountability}, 
we have proved that $\P M_H$ is open in $\P M$.

By \eqref{eqn:VKerBVB}, we have
$V(\cB^{(l)})'_{s}=
V(k_{s-1})'\oplus\cdots\oplus V(k_s-1)'$.
Let $r_0+\cdots+r_{k_s+j-1}< k\leq r_0+\cdots+r_{k_s+j}$
with $k_s+j<k_{s+1}$.
Note that
\begin{eqnarray*}
&&
\bigwedge^{r_0} V(0)'\otimes \bigwedge^{r_1} V(1)'\otimes
\cdots\otimes \bigwedge^{r_{k_s-1}} V(k_s-1)'\\
&=&
\bigwedge^{r(\cB^{(l)})_0} V(\cB^{(l)})'_0\otimes \bigwedge^{r(\cB^{(l)})_1} 
V(\cB^{(l)})'_1\otimes
\cdots\otimes \bigwedge^{r(\cB^{(l)})_{s}} V(\cB^{(l)})'_s.
\end{eqnarray*}

Let $k':=k-(r_0+\cdots+r_{k_s-1})$.
On $V(\cB^{(l)})'_{s+1}$, some nonzero scalar multiples of
$$
\bigwedge^{k'} B^{(l)}_{s+1}
$$
converge to
the restriction on
$$
\bigwedge^{r_{k_s}} V(k_s)'\otimes \bigwedge^{r_{k_s+1}} 
V(k_s+1)'\otimes
\cdots\otimes \bigwedge^{k'-(r_{k_s}+\cdots+r_{k_s+j-1})} 
V(k_s+j)'
$$
of
$$
\bigwedge^{k'} ({A_{k_s}}_{| V(k_s)'}\oplus 
\cdots \oplus {A_m}_{| V(m)'})
$$
by the definition of $A_{k_s}, A_{k_s+1},\ldots,A_{k_s+j}$
for $k_s+j< k_{s+1}$ (see the arguments before Lemma \ref{thm:expansion}).

Hence the
restrictions on
$$
\bigwedge^{r_0} V(0)'\otimes \bigwedge^{r_1} V(1)'\otimes
\cdots\otimes \bigwedge^{k-(r_0+\cdots+r_{k_s+j-1})} V(k_s+j)'
$$
of some nonzero scalar multiples of
$$
\bigwedge^{k} ({B^{(l)}_0}_{| V(\cB^{(l)})'_0}\oplus
\cdots \oplus {B^{(l)}_p}_{| V(\cB^{(l)})'_{p}}),
$$
converge to $\bigwedge^k\cA$.

We have thus seen that $\overline{\lambda}(\P\cB^{(l)})\to
\overline{\lambda}(\P\cA)$.
Hence we have proved that
$\overline{\lambda}$ is continuous,
thanks to Proposition \ref{prop:2ndCountability}.
\end{proof}

\begin{corollary}
\label{cor:PM_HtoHingeConti}
$\P\varphi: \P{M}_H\to \Hinge^*(V)$
is continuous.
\end{corollary}

\begin{proof}
Neretin defined a mapping $\lambda^\circ: \Hinge^*(V)\to
\prod_{k=1}^{n-1}\P\End(\bigwedge^k V)$, and proved
that
$\lambda^\circ$ is a closed embedding \cite[Theorem 4.4]{Neretin98}.
Since we have $\overline{\lambda}=\lambda^\circ\circ \P\varphi$,
$\P\varphi$ is continuous by Theorem \ref{thm:lambdaConti}.
\end{proof}

\begin{example}
Let
$$
A_0:=
\begin{bmatrix}
1 & 0 & 0 & 0\\
0 & 1 & 0 & 0\\
0 & 0 & 0 & 0\\
0 & 0 & 0 & 0\\
\end{bmatrix},
\quad
A_1:=
\begin{bmatrix}
0 & 0 & 0 & 0\\
0 & 0 & 0 & 0\\
0 & 0 & 1 & 0\\
0 & 0 & 0 & 1\\
\end{bmatrix},
$$
and $A(\epsilon):=A_0+\epsilon A_1$.
Then
$$
\lim_{\epsilon\to 0}\P A(\epsilon)= \P\pi(A_0, A_1)\in \P M_H.
$$
We have $p=0, k_0=2$, $V(0)'=\langle \e_1,\e_2\rangle$, and
$V(1)'=\langle \e_3,\e_4\rangle$.
Then
$$
\bigwedge^3 A(\epsilon):
\left\{
\begin{array}{l}
\e_1 \wedge \e_2 \wedge \e_3 \mapsto 
\epsilon \e_1 \wedge \e_2 \wedge \e_3\\
\e_1 \wedge \e_2 \wedge \e_4 \mapsto 
\epsilon \e_1 \wedge \e_2 \wedge \e_4\\
\e_1 \wedge \e_3 \wedge \e_4 \mapsto 
\epsilon^2 \e_1 \wedge \e_3 \wedge \e_4\\
\e_2 \wedge \e_3 \wedge \e_4 \mapsto 
\epsilon^2 \e_2 \wedge \e_3 \wedge \e_4\\
\end{array}
\right.
$$
Hence
$$
\lim_{\epsilon\to 0}\epsilon^{-1}\bigwedge^3 A(\epsilon)
=id_{\langle \e_1 \wedge \e_2 \wedge \e_3, 
\e_1 \wedge \e_2 \wedge \e_4\rangle}.
$$

We have
$$
\bigwedge^3\pi(A_0, A_1)
=\left(
\bigwedge^3 ({A_0}_{|V(0)'}\oplus {A_1}_{|V(1)'})
\right)_{|\bigwedge^2V(0)'\otimes \bigwedge^1V(1)'}.
$$
Since ${A_0}_{|V(0)'}=id_{V(0)'}$, ${A_1}_{|V(1)'}=id_{V(1)'}$,
and $\bigwedge^2V(0)'\otimes \bigwedge^1V(1)'
=\langle \e_1 \wedge \e_2 \wedge \e_3, 
\e_1 \wedge \e_2 \wedge \e_4\rangle$,
we have
$$
\bigwedge^3\pi(A_0, A_1)
=
\left(
\bigwedge^3 id_V
\right)_{|\langle \e_1 \wedge \e_2 \wedge \e_3, 
\e_1 \wedge \e_2 \wedge \e_4\rangle}.
$$
Hence
$$
\bigwedge^3\pi(A_0, A_1)
=\lim_{\epsilon\to 0}\epsilon^{-1}\bigwedge^3 A(\epsilon).
$$
\end{example}


\end{document}